\documentclass[reqno]{amsart}
\usepackage{amsfonts,amssymb,amsmath}
\usepackage{graphicx}
\usepackage{tikz}
\usepackage{orcidlink}

\newtheorem{prop}{Proposition}

\newtheorem{thm}{Theorem}

\usepackage{graphicx}
\usepackage[dvips]{psfrag}
\usepackage{subfigure}

\begin{document}

\title[Global dynamics and integrability of  a Leslie-Gower  model]
 {Global dynamics and integrability of  a Leslie-Gower  
predator-prey model with linear functional response and  generalist predator}

\author[Alvarez--Ram\'{\i}rez, Garc\'ia--Salda\~na, Medina]{Martha Alvarez--Ram\'{\i}rez$^{1\orcidlink{0000-0001-9187-1757}}$,  Johanna  D.  Garc\'ia--Salda\~na$^{2\orcidlink{0000-0002-4475-5064} }$ and Mario Medina$^{1 \orcidlink{0000-0003-0295-5145}}$}
\address{$^{1}$ Departamento de Matem\'aticas\\ UAM--Iztapalapa\\09310 Iztapalapa\\ M\'exico City,
M\'exico.}

\address{$^{2}$ Departamento de Matem\'{a}tica y F\'isica Aplicadas,  Universidad Cat\'olica de la Sant\'isima Concepci\'on, 
Alonso de Ribera 2850, Concepci\'on, Chile.}
\email{mar@xanum.uam.mx, jgarcias@ucsc.cl, mvmg@xanum.uam.mx}



\maketitle

\begin{abstract}
We deal with a Leslie-Gower predator-prey model with a generalist or alternating food for predator
and  linear functional response. Using a topological equivalent polynomial system we prove  that the system is not Liouvillian  (hence also not Darboux) integrable. 
In order to study the global dynamics of this model, we  use the Poincar\'e compactification of $\mathbb{R}^2$ to  characterize all
 phase portraits in the  Poincar\'e disc, obtaining two different topological phase portraits.
\end{abstract}

\section{Introduction and statement of results}
The problem of investigating  the dynamics of biological systems has attracted the attention of many researchers due to its inherent importance, in particular, many of the scientists have aimed to study the interaction between two animal species where  one of the species preys upon the other in an ecological niche. 
One of the first models to incorporate predator-prey interactions was proposed in 1925 by   Lotka and Volterra.  More recent models to study predator--prey interrelation are those named Leslie-Gower type models. Next, we comment on several aspects that are considered to construct a Leslie-Gower type model for a  predator-prey interaction.
In two works, one by Leslie \cite{Leslie1958}, and later in a joint work with  Gower \cite{Leslie1960}, they  were the first researchers that realized that the existing interference among  predators and prey plays a key role in modeling their interaction. They modeled this by assuming that the contraction in the number of predators   is related to the availability of their preferred food, and this relation is inversely proportional to such availability per individual. Even more, in the case that the number of predators lead to their extinction  when there are no enough prey to feed the predators, it is possible to add up a  new term in the denominator, usually named as extra food. 
Another aspect to take into consideration to get the model  is the way a predator  consumes its prey, since this aspect regulates the dynamics of the entire system. This is introduced  via a function that is named as the functional response. 

In \cite{Das2024}, Das and coauthors   added  the Allee effect  and simultaneous harvesting to a  Holling-Tanner predator-prey model and  showed that under specific hypothesis there are at most three equilibrium points contained in the first quadrant, also both populations,  prey and predators may  coexist if some conditions are fulfilled  and  showed the existence of saddle-node and Hopf bifurcations.  Arancibia-Ibarra and Flores \cite{IbarraF2020} included an Allee effect (strong nad weak)  to a May-Holling-Tanner predator-prey model and  proved that the extinction of both populations is allowed and the existence of  an alternative food for the predator lead to coexistence of both populations, or  the extinction of the prey and also that the prey and predator may coexist with simultaneous oscillating populations. Valenzuela and coauthors \cite{Valenzuela2020} included prey defense as well as a generalist predator in a Leslie-Tanner predator-prey model and gave conditions on the parameters so  that Bautin, Hopf and simultaneous supercritical Hopf bifurcations exist.

In deterministic systems, the presence of first integral or constant motion is known to be responsible for the regular evolution of the system's phase space orbits in well-defined regions of the phase space.
For planar differential systems, the existence of  integrals of motion  (hereafter referred to as the {\em first integral}) presupposes a regular evolution of the system's phase orbits in a well-defined region of phase space, so that the system can be integrated. Therefore,  its phase portrait can be determined, since the system can be integrated. The interested reader is referred to \cite{Chavarriga}. 

In recent years, interest in studying the integrability of differential equations has attracted much attention from the mathematical community.
Darboux's integrability theory plays a central role in the integrability of polynomial differential systems, 
because it provides sufficient conditions for achieving integrability within large family of functions.
We emphasize that it works for real or complex polynomial differential systems and that the study of complex algebraic solutions is necessary to obtain all the real first integrals of a real polynomial differential system.
In particular, if a polynomial differential system has an elementary first integral, then this integral can be computed by using the invariant algebraic curves of the system. The Darboux theory of integrability plays a central role in the integrability of the polynomial differential systems. 
This theory provides sufficient conditions to obtain the first integral using  invariant algebraic curves, also called the {\em Darboux polynomial}.
Roughly speaking,  we may use Darboux integrability theory to  give the conditions of integrability within a family of Liouville functions (that is, functions obtained from complex rational functions by finite processes of integration, exponentiation, and algebraic operations). In fact, Darboux's theory of integrability can be applied to real or complex polynomial differential systems. Nevertheless,  finding the first real integral of a real polynomial differential system requires the study of complex algebraic solutions.  Appendix  \ref{ap_inregra} contains all the results related to Darboux integrability theory required for this paper. The reader can also check the details in the  Chapter 8 of  \cite{Dumortier}.

There are models described by differential equations where one finds solutions that are unbounded, in such a way that to understand them it is necessary to study the vicinity of infinity, in such a way that we can know the behavior of orbits that reach or escape infinity. 
 In the special case in which the associated vector field is polynomial, the global study of the system is carried out through the named Poincar\'e compactification technique, which is an extension of the field to a compact manifold that allows to draw global phase portraits in the Poincar\'e disc.
Strictly speaking, the Poincar\'e disc is a closed disc of radius one centered at the origin of $\mathbb{R}^2$, 
whose interior  has been identified with the whole plane,  and its border is identified with the infinity of $\mathbb{R}^2$.
We would highlight that in a plane we can go to infinity in as many directions as there are points on the circle.
However, there is a unique analysis method that extends the polynomial differential system defined in $\mathbb{R}^2$ to the Poincar\'e disc.
This extended system defined by the Poincar\'e disc allows us to study how the trajectory of a polynomial differential system approaches and moves away from infinity. For further details the reader is addressed to  Appendix \ref{ap_compac} included in this paper or Chapter 5  of~\cite{Dumortier}.

Gonz\'alez-Olivares and Rojas-Palma  in \cite{Gonzalez2020} studied a modified Leslie-Gower type predator-prey model 
with generalist predator as follows:
\begin{equation}\label{eq1}
\begin{array}{l}
 \dfrac{dx}{dt} = rx\left(1-\dfrac{x}{k}\right)-qxy, \vspace{0.2cm}\\
 \dfrac{dy}{dt} = sy\left(1-\dfrac{y}{nx+c}\right),
\end{array}
\end{equation}
where $x$ represents the population of prey and $y$ represents the population of predator.
All the parameters $r$, $k$, $q$, $s$, $n$ and $c$ are positive and their
biological meanings are summarized in Table \ref{tabla1}. 
Actually, they proved that the only positive equilibrium point of \eqref{eq1}, when it exists, is globally asymptotically stable, and  also constructed a suitable Dulac function to prove the nonexistence of periodic orbits.  This means, that 
the solution orbits associated with this point are unbounded.
\begin{table}[hbt]
\begin{tabular}{ |c | l |}
\hline
Parameter &  Description \\\hline
$r$ & intrinsic prey growth rate \\ \hline
$k$ & prey environmental carrying capacity \\ \hline
$q$  & consuming maximum rate per capita of the predators \\\hline
$s$ & intrinsic predator growth rate \\\hline
$n$ & food quality \\\hline
$c$ & amount of alternative food available for predators  \\ \hline
		\end{tabular}
\caption{Biological meanings of parameters in system \eqref{eq1}.}\label{tabla1}
	\end{table}

This system \eqref{eq1} is a  Kolmogorov-type system, so  the axes are invariant
(forwards and backwards in time), see  e.g.  \cite{Dumortier, Freedman}. So,  it  is defined in the first quadrant $\Omega = \{ (x,y)\in\mathbb{R}^2 \mid x\geq 0, \; y\geq 0\}$.

To the best of our knowledge, up to now,  no researcher has yet investigated the integrability of  system  \eqref{eq1}
in the Liouville sense.  Furthermore, the description of the global dynamics in Poincar\'e disc has also not been given.
Therefore, we aim to investigate them using Darboux functions and Poincar\'e compactification, respectively.

Nowadays, the Poincar\'e compactification  has been used by several authors to study the global dynamics of
some versions of prey-predator models.
 Although not exhaustive, the following works can be mentioned.
Llibre and Valls \cite{VallsL}  considered the classic Leslie-Gower model, i.e., without generalist predator, 
 performing a  complete description of its phase portraits, depending on the parameters, in the Poincar\'e disc modulo topological equivalence.
 In \cite{Diz-Pita2}, Diz-Pita   et al.  classified the global dynamics of a 
 Rosenzweig and MacArthur system, which is a   predator-prey systems with functional response Holling type II.
As well, paper \cite{Diz-Pita} by  Diz-Pita deals with a  prey-predator model  including an immigration term in both species,  i.e.,  a system with a Holling type I functional response, where global dynamics is performed by studying global phase portraits in the positive quadrant of the Poincar\'e disc.

At this point, before continuing with our objectives, we should point out that  system \eqref{eq1} is non-polynomial, so the 
previous methods regarding to  integrability and Poincar\'e compactification cannot be applied directly.
This problem will be solved by  reparameterizing  the variables and time.
In order to activate this, we shall follow the approach used in several different types of prey-predator 
(see, for instance \cite{VallsL}, \cite{saez}), to  transform system \eqref{eq1} to a topologically equivalent system. We achieve this by introducing
 the dimensionless variables and the time rescaling, given by the function
 $\psi: {\tilde\Omega}\times\mathbb{R} \to \Omega\times \mathbb{R}$, so that
$$
 \psi (u,v,\tau)= \left(  ku,knv,\dfrac{u+ \frac{c}{nk}}{r}\tau\right) = (x,y,t),
$$
where ${\tilde\Omega} = \{ (u,v)\in\mathbb{R}^2 \mid 0\leq u\leq 1, \; v\geq 0\}$. 
The mapping $\psi$ is a diffeomorphism for $(u,v)\neq (0,0)$  that preserves the orientation of time, since  
$\det \big(D\psi (X,Y,\tau) \big)= \dfrac{k^2 n}{r}\left(u+\dfrac{c}{nk} \right)>0$, see  \cite{Dumortier}.
Hence, in the new coordinates,  system \eqref{eq1} is   topologically
equivalente to the following  polynomial system, where we still denote $\tau$ by $t$ and
rename the variables $(u,v)$  as $(x, y)$:
\begin{equation}\label{eq22}
\begin{array}{l}
 \dfrac{dx}{dt}= u\left( \dfrac{c}{kn}+x\right)\left(1-x- \dfrac{knq}{r}y\right), \vspace{0.2cm}\\
 \dfrac{dy}{dt}=\dfrac{s}{r} y \left( \dfrac{c}{kn} + x-y\right).
\end{array}
\end{equation}

Next, we set $A=\dfrac{knq}{r}$, $B=\dfrac{s}{r}$ and $C=\dfrac{c}{kn}$,  all of which are positive. 
Using these variables, \eqref{eq22} is then transformed into the following equivalent polynomial system:
\begin{equation}\label{eq4}
\begin{array}{l}
\dfrac{dx}{dt}= x(C+x)(1-x-Ay), \vspace{0.2cm}\\
\dfrac{dy}{dt}=By (C+x-y).
\end{array}
\end{equation}

Now, we are able to establish our main result on the Liouville integrability of the  differential system \eqref{eq4}, 
whose proof will be given in Section \ref{sec_teo1}.

\medskip
\begin{thm}\label{Theo1}  System \eqref{eq4} is not Liouvillian integrable. \end{thm}

 In \cite{Gonzalez2020}, the authors proved that 
 the interior equilibrium point,   hereinafter denoted by  $E^*$,  is a global attractor in this open quadrant of $\mathbb{R}^2_+$, but it is  unknown where the orbits attracted by $E^*$ are born. Next theorem provides  the global topological phase portraits of  system \eqref{eq4}.
 
\begin{thm}\label{Theo2} 
The global phase portrait of system \eqref{eq4}, in the closed positive quadrant of the Poincar\'e disc, is topologically equivalent to one of the two phase portraits presented in Figure~\ref{fig:global1}. 
\begin{figure}[h]
\centering
\subfigure[$1-AC\leq 0$]{
\includegraphics[width=4cm]{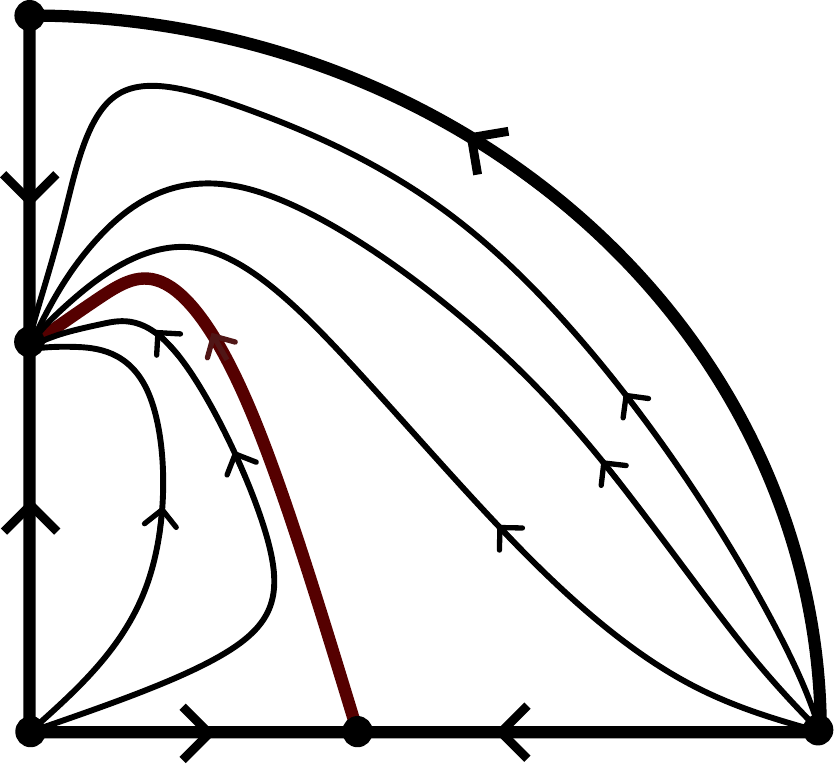}\hspace{0.8cm}}
\subfigure[$1-AC>0$]{
\includegraphics[width=4cm]{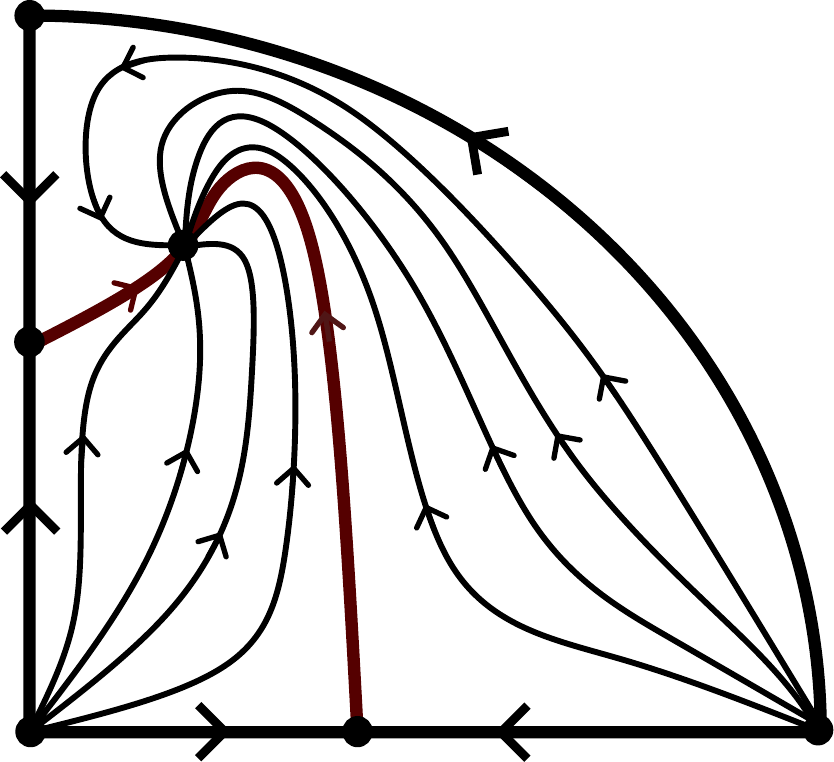}}
\caption{Global phase portrait of system \eqref{eq4} at infinity,  shown in  the closed positive quadrant of the Poincar\'e disc.}
\label{fig:global1}
\end{figure}
\end{thm}
Theorem  \ref{Theo2} will be  proved in Section \ref{sec_teo2}.

The two main objectives of this article are  mainly mathematical, but they certainly make biological sense.
These help explain the dynamic behavior of the generalist predator-prey system in terms of specific parameters.
More precisely, from a biological point of view, the results strongly confirm that alternative foods contribute to the conservation or extinction of species.

\section{Proof of the Theorem  \ref{Theo1}}\label{sec_teo1}
Since Liouville integrability is based on the existence of invariant algebraic curves and their multiplicity through the exponential factors,
in order to  prove the  non-integrability in the Louville sense, we need to characterize its polynomial first integrals, Darboux polynomials and exponential factors.

From \cite{chris}, we know   that the existence of exponential factors $\exp (g/f)$ 
is due to the fact that the multiplicity of the invariant algebraic curve $f = 0$ is larger than~1, and when $f$ is constant, the existence 
of the exponential factor depends on  the multiplicity of the straight line at infinity in the projective plane.
So, we shall check those conditions.

We denote by 
$$
X = x(C+x)(1-x-Ay) \frac{\partial}{\partial x} + By (C+x-y) \frac{\partial}{\partial y}
$$
the vector field defined by system  \eqref{eq4}.
We begin by noting that from \eqref{eq4} is satisfied  that $\dot{x}\mid_{x=0}=0$, $\dot{x}\mid_{x+C=0}=0$ and $\dot{y}\mid_{y=0}=0$.
Hence,   $f_1=x$, $f_2=x+C$ and $f_3=y$ are invariant algebraic curves of the 
 differential system \eqref{eq4}. 

 Now we must find the respective cofactors associated to the curves $f_j=0$, $j=1,2,3$.
Seeking the cofactor expression  we shall use the definition \eqref{eq2_apint}.
Then the corresponding cofactors for $x=0$, $x+C=0$  and $y=0$ are given by 
\begin{align*}
& K_1 = -(C+x) (-1+x+ Ay),\\
& K_2 =-x(-1+x+ Ay),\\
& K_3= B(C+x-y).
\end{align*}

Now, we will  discuss the algebraic multiplicity of the invariant algebraic curves. 
We have that the  extactic curve ${\mathcal E}_1$ of $X$ is described by
{
\begin{align*}
{\mathcal E}_1 =  & -B x (C + x) y \big[C^2(1 - B) + C (2 - 2 B - 3 C + BC) x  - C (1 - 3 B + 2 AC) y\\
& \;  + (1 - B - 6 C + 2 BC + 2 C^2) x^2  -(2 - 3 B - 3 C + 4 AC + 3 BC - 3 A C^2) x y \\
& \; -(B - AC) (2 + A C) y^2 - (3 - B - 4 C) x^3  +  (5 - 2 A - 3 B - 2 C + 6 AC) x^2 y    \\
& \; + (4 A + 2 B - AB - 3 AC + 2 A^2 C) x y^2+ 2 x^4 + 3 (A-1) x^3 y   + (A-5) A x^2 y^2 \\
& \; -A ( AC-B  + 2 A x) y^3
\big].
\end{align*}}
This implies that each algebraic curve given  $x = 0$, $x+C=0$ and $y =0$ is a power factor one of ${\mathcal E}_1$.
Consequently,  the multiplicity of these invariant algebraic curves  is 1, and it  turns out that  they do not give rise to exponential factors of the form $\exp (g/f)$. However, we still have to check if there are exponential factors coming from the invariant curve at infinity.   

Since the degree of the system \eqref{eq4} is 3, the cofactor has at most degree 2.
In what follows, we consider the algebraic multiplicity of an invariant algebraic curves.
We first consider the case $F= \exp (a_{00}+ a_{10} x+ a_{01}y)$. By using the 
 definition of   exponential factor, we  arrive at   (see Appendix \ref{ap_inregra}):
$$
 x(C+x)(1-x-Ay) \frac{\partial F}{\partial x}  +  By (C+x-y)   \frac{\partial F}{\partial y} = K_4(x,y)  F(x,y),
$$
but the cofactor $K_4(x,y)$ must have at most degree 2, this enforces that $a_{10} =0$. Thus  $F= \exp (a_{00}+  a_{01}y)$
with cofactor $K_4(x,y) = B  a_{01} (C + x - y) y $.

Now,   we  shall check  if  there  is  a  second  exponential  factor  having the form
$G= \exp (a_{00}  + a_{01}y + a_{20} x^2+ a_{11} x y+a_{02}y^2)$  coming  from   the   invariant   curve   at   infinity.
Repeating the above  process and taking into account that  cofactor  have  degree  at  most  2, we obtain 
that $a_{20}=a_{11}=a_{02}=0$. Therefore, $G= \exp (a_{00}  + a_{01}y)$
and  we  get  the  previous  exponential  factor. So  there  are  no  additional  exponential  factors.

The next step is to apply the statement (a) of Theorem \ref{teo_app} in the Appendix   \ref{ap_inregra} (see also statement (i) of Theorem 8.7 of \cite{Dumortier}), there exists a linear combination between the factors $K_j$ if and only if the first Darboux integral exists. 
The linear combination becomes
$$
\lambda_1  K_1 + \lambda_2 K_2 + \lambda_3 K_3 + \lambda_4 K_4 =0,
$$
where $\lambda_j\in {\mathbb C}$, $j=1,2,3,4$. By some calculations we get that this equation is satisfied if only if 
$\lambda_1 = \lambda_2 = \lambda_3= \lambda_4=0$  or 
 $\lambda_1 = \lambda_2 = \lambda_3= 0$ with $a_{01}=0$ and $\lambda_4\in {\mathbb C}$.
Since $a_{01}=0$, we have that  $F(x,y)=\exp ( a_{00})$.
We noticed that this is not a  Darbouxian integrating factor because it is a constant function.

Continuing in the same spirit, we proceed to calculate the divergence
of system~\eqref{eq4}. This becomes  
$$\text{div} (X) = (1 + B) C  - (2 B + A C) y +  (2 + B - 2 C)x - 3 x^2- 2 A xy.$$
According to statement (b) of Theorem \ref{teo_app} of the Appendix \ref{ap_inregra}, there is a linear combination between the cofactors 
$K_j$, $j=1,2,3,4$, and the  divergence of the system, if and only if there exists a Darboux integrating factor \eqref{sys_pq2}. But,
we can easily check that
the linear combination of the cofactors and the divergence of the vector field
$$
\lambda_1  K_1 + \lambda_2 K_2 + \lambda_3 K_3  + \lambda_4 K_4 + \text{div} (X) =0,
$$
 cannot be satisfied for $\lambda_j$, not all zero. So, there is no such a  linear combination.

In summary, from the calculations above, the Theorem \ref{teo_app} in Appendix   \ref{ap_inregra},
we conclude that the differential system \eqref{eq4}  does not have any Darboux function as integrating factor.
In view of  Theorem \ref{teo_darint}  of Appendix \ref{ap_inregra},  system \eqref{eq4} does not have any Liouvillian first integral.
This concludes the proof of theorem \ref{Theo1}.

\section{Proof of the Theorem  \ref{Theo2}}\label{sec_teo2}
In order to prove Theorem  \ref{Theo2} we will study the behavior of the finite and infinite equilibria.
Indeed, it is not necessary to study the existence of limit cycles, since in \cite{Gonzalez2020} the authors proved that they do not exist.

\subsection{The finite equilibrium points}\label{pts1}
As we remarked before, the system  \eqref{eq4}  is a Kolmogorov type, then the $x$-axis and $y$-axis are invariant sets.

\begin{prop}
The set 
${\tilde\Omega} = \{ (x,y)\in\mathbb{R}^2 \mid 0\leq x\leq 1, \; y\geq 0\}$
is a positively invariant and attracting region for the flows of  system \eqref{eq4} in the first quadrant.
\end{prop}

\begin{proof}
Since \eqref{eq4} is a Kolmogorov type, the $x$--axis and the $y$--axis are invariant sets. Additionally, if  $x=1$ we  get that
$dx/dt= -A(1+C)y<0$ and whatever it is the sign of $dy/dt=By(C-y)$
the trajectories enter and stay in the region $\tilde{\Omega}$.
\end{proof}

A straightforward computation shows that system \eqref{eq4} has
the trivial equilibrium point $E_0 = (0, 0)$,  the prey-free equilibrium point $E_1 = (0,C)$ and the
predator-free equilibrium point $E_2 = (1,0)$. The only  co-existence or interior equilibrium is   $E^*(x_*,y_*) = (\frac{1- AC}{1+A},\frac{1+C}{1+A})$ with $1- AC >0$.

Now, the local stability of the above equilibrium points is investigated using the  Jacobian matrix, namely
\begin{equation}
J=\left(\begin{array}{cc}
x(2-3x-2Ay)-C(-1+2x+Ay) & -Ax(C+x)  \\
By&  B(C+x-2y) 
\end{array}\right).\label{jacob}
\end{equation}

\begin{prop}
The equilibrium point $(0,0)$ is  a repeller point, while point $(1, 0)$ is a  saddle point.
\end{prop}

\begin{proof}
The Jacobian matrix \eqref{jacob} at $E_0$ can be written as 
$$J (0,0)=\left(\begin{array}{cc}
C & 0  \\
0&  BC 
\end{array}\right).$$
Hence the eigenvalues  are $C>0$ and $BC>0$. Therefore, the equilibrium $(0,0)$ is a  repeller point.
Similarly, the Jacobian matrix \eqref{jacob} evaluate at the equilibrium point $(1, 0)$ becomes
$$
J (1,0)= \left(\begin{array}{cc}
-1-C & -A(1+C)  \\
0 &  B(1+C) 
\end{array}\right),$$
with eigenvalues $\lambda_1= -1-C <0$ and $\lambda_2=B(1+C)>0$. Thus  $(1, 0)$ is
a saddle point. 
\end{proof}

\begin{prop}
The equilibrium point $E_1=(0, C)$  is
\begin{enumerate}
\item[{1.}] an  attractor point if  $1-AC<0$;
\item[{2.}] a  saddle point if  $1-AC>0$;
\item[{3.}] a  saddle-node if  $1-AC=0$.
\end{enumerate}
\end{prop}

\begin{proof}
For the point $(0, C)$  matrix \eqref{jacob} goes over
$$
J (0,C)= \left(\begin{array}{cc}
-C (-1+AC) & 0  \\
BC &  -BC 
\end{array}\right).$$
Hence  $\lambda_1= -BC<0$ and $\lambda_2 = C(1-AC)$
are the eigenvalues of the Jacobian matrix at $(0,C)$.
\begin{enumerate}
\item[{1.}] If $1-AC<0$, then $\lambda_1<0$ and $\lambda_2<0$, so the equilibrium point is an  attractor point.
\item[{2.}] If  $1-AC>0$, then  $\lambda_1<0$ and $\lambda_2>0$, so the equilibrium point $E_1$ is a hyperbolic saddle.
\item[{3.}] If  $1-AC=0$, then $\lambda_1<0$ and $\lambda_2=0$ and the  equilibrium points $(0,C)$ and $E^*$
collapse. 
In order to  determine the nature  of $(0,C)$, we linearize system \eqref{eq4} at $E_1$ and diagonalize the linear part, then
the center manifold theorem will be applied, see e.g.  \cite{Dumortier, perko}.
Setting $(x, y)= (\xi,\eta+C)$, system \eqref{eq4} is translated to the origin of coordinate,
 then we have the following system
 \begin{equation*}\label{eq3a}
\begin{array}{l}
 \dfrac{d\xi}{dt}=  -BC \xi -B\xi^2+(AC-B) \xi\eta +A\xi\eta^2+(C-1-2 A C)\eta^2 + (1+A)\eta^3, \vspace{0.2cm}\\
 \dfrac{d\eta}{dt}= (1-C-2AC)\eta^2-AC\xi\eta - (1+A)\eta^3-A\xi\eta^2.
\end{array}
\end{equation*}

Since $1-AC=0$ and $C>0$, it follows that $1-C-2AC\neq 0$. 
Then, the non-hyperbolic equilibrium point $E_1=(0,C)$ is an attractor saddle-node.
\end{enumerate}
\end{proof}

\begin{prop} 
For system \eqref{eq4}, if there exists exactly one simple coexistence equilibrium $E^*(x_*,y_*) = (\frac{1- AC}{1+A},\frac{1+C}{1+A})$, it is an attractor point.
\end{prop}
 \begin{proof}
 In the following analysis, we consider $1 - A C>0$.
 The Jacobian matrix in $E^*$ is
 $$
 J (E^*)=\left(\begin{array}{cc}
\dfrac{(1+C)(-1+AC)}{(1+A)^2} & \dfrac{A(1+C)(-1+AC)}{(1+A)^2} \vspace{0.3cm} \\
\dfrac{B(1+C)}{(1+A)} &  -\dfrac{B(1+C)}{(1+A)}
\end{array}\right).
 $$
The two eigenvalues of $ J (E^*)$ are
 $$ \lambda_{1,2} = \frac{-\eta \pm \sqrt{\eta^2+4 B(1+A)^2(AC-1)}}{2(1+A)^2},$$
 where $\eta = -1-(1+A)B+AC$.
 
We can take advantage of the fact that the determinant of a matrix is the product of the eigenvalues, and the trace of the matrix  is the sum of its eigenvalues: 
$\det (J (E^*)) = \dfrac{B(1+C)^2(1-AC)}{(1+A)^2}>0$ and $\text{trace}(J (E^*)) = \dfrac{(1+C)(-1-(1+A)B+AC)}{(1+A)^2}<0 $.
This  means that both eigenvalues  $\lambda_{1,2}$  have  negative sign. 
So $E^*$ is locally  stable.
\end{proof}

At this step, we would point out  that in the work  \cite{Gonzalez2020}, the authors adapted the idea of Korobeinikov \cite{Korobeinikov} 
to prove   the global asymptotic stability of  the coexistence equilibrium point  of  system \eqref{eq1} by using a suitable Lyapunov function. It is clear, therefore, since  systems  \eqref{eq1} and  \eqref{eq4} are topologically equivalent, it follows  that the   equilibrium point $E^*(x_*,y_*) = (\frac{1- AC}{1+A},\frac{1+C}{1+A})$ is also globally asymptotically stable.

\subsection{The infinite equilibrium points}\label{pts2}
In order to characterize the global phase portrait of system (\ref{eq4})  on the Poincar\'e disc (see Appendix \ref{ap_compac}), we will examine the infinite equilibrium points.

In the first local coordinate chart $(U_1,\phi_1)$ system \eqref{eq4}  is written as
\begin{align}\label{Poincare1}
&\dot{u} =u[(1+Au-v)(1+Cv)+Bv(1-u+Cv)],\\
& \dot{v}=v(1+Au-v)(1+Cv),\nonumber
\end{align}
where the dot denotes the derivative with respect to $t$.
We have that both, $u$-axis and $v$-axis are invariant under the flow of system \eqref{Poincare1}. In consequence the first quadrant is an invariant set.
The   points $ (0,0)$ and $(-1/A,0)$ are the infinite equilibria.  
Since we are only  interested in studying the dynamics in  ${\mathbb R}^2_+$, the point $(-1/A,0)$ is discarded.
On the other side, the linear part of system \eqref{Poincare1} at the origin is the identity matrix, so $(0,0)$ is an unstable node.

In the  local chart $U_2$ system \eqref{eq4} becomes
\begin{align}\label{Poincare2}
&\dot{u} =-u \big[A u +(A C- B) v+ u^2 + (B+ C-1) uv + (B C -C)v^2\big],  \\
&\dot{v} =-Bv^2(-1+ u + Cv).\nonumber
\end{align}
Here once again the dot denotes the derivative with respect to $t$. In the following lemma we study the behavior around the origin of the solutions of the system.

\begin{prop}
The local phase portrait at the origin of system \eqref{Poincare2} has two parabolic sectors and two hyperbolic sectors  as is showed in Figure~ \ref{fig:U2}.
\begin{figure}[h!]
\begin{center}
\includegraphics[width=5cm]{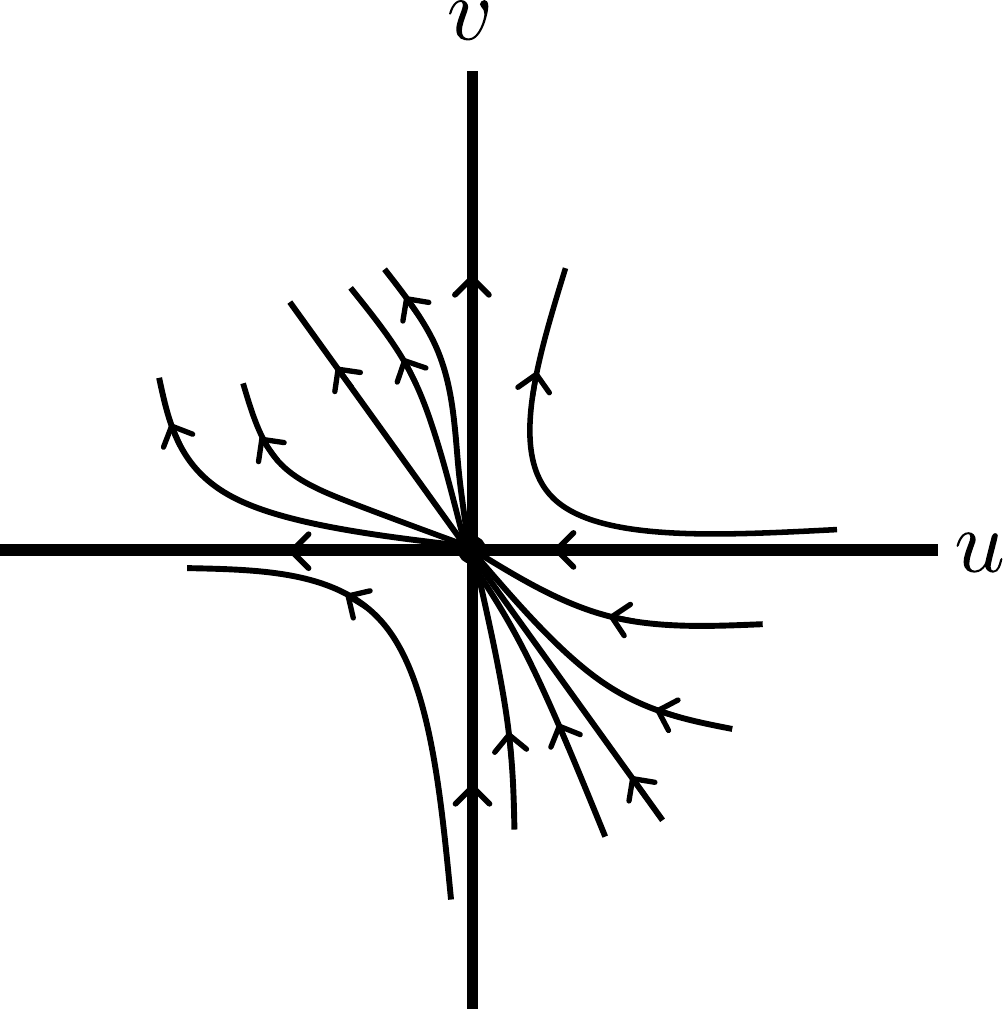}
\end{center}
\caption{Topological local phase portrait of system \eqref{Poincare2}.}
\label{fig:U2}
\end{figure}
\end{prop}

\begin{proof}

It follows that the  origin $(0, 0)$ of the chart $U_2$ is an infinite equilibrium point.
Because the linear part of \eqref{Poincare2} is the identically
zero matrix,  then the local phase portrait at $(0,0)$ is determined doing variables changes called blow-ups, see~\cite{jarque}.

In order to get the complete local phase portrait around the origin we need to do both directional blow-ups since $u=0$ and $v=0$ are both characteristic directions (for more details about characteristic directions at an equilibrium point see \cite{andronov}).
We proceed to carry out the $u$-directional  blow-up given by the change of variables $(u, v)= (u, u w)$ and, also we  scale the time variable of the system by $dt/d\tau=u$ to eliminate the common factor $u$. Therefore system \eqref{Poincare2} goes over 
\begin{align}\label{blowup-u}
& \dot{u} = -u [ A + u + (AC-B) w + (B + C-1) u w + C(B- 1 )uw^2], \nonumber\\
& \dot{w} = A w + uw+A C w^2 + uw^2- C uw^3, 
\end{align}
where  the dot denotes derivative with respect to the new time $\tau$.
This system on the straight line $u = 0$ has  the  equilibrium points $(0, 0)$ and $(0, -1/C)$.
The Jacobian matrix of \eqref{blowup-u} at the origin  has as eigenvalues  $A $ and $-A$.
Then,  the equilibrium $(0, 0)$ is a hyperbolic saddle. On the other side, the eigenvalues of the Jacobian matrix at $(0, -1/C)$ are $-A<0$ and $-B/C < 0$, so it is a hyperbolic stable node. 

Going back the blow-up process and undoing the scale change, where we  swap the second quadrant for the third quadrant, and vice versa.
We conclude that the  local phase portrait of the straight line $u = 0$  of system \eqref{Poincare2}  is topologically equivalent to that given 
in  Figure \ref{figure1}(b).
\begin{figure}[h!]
	\centering
	\subfigure[]{\includegraphics[scale=0.3]{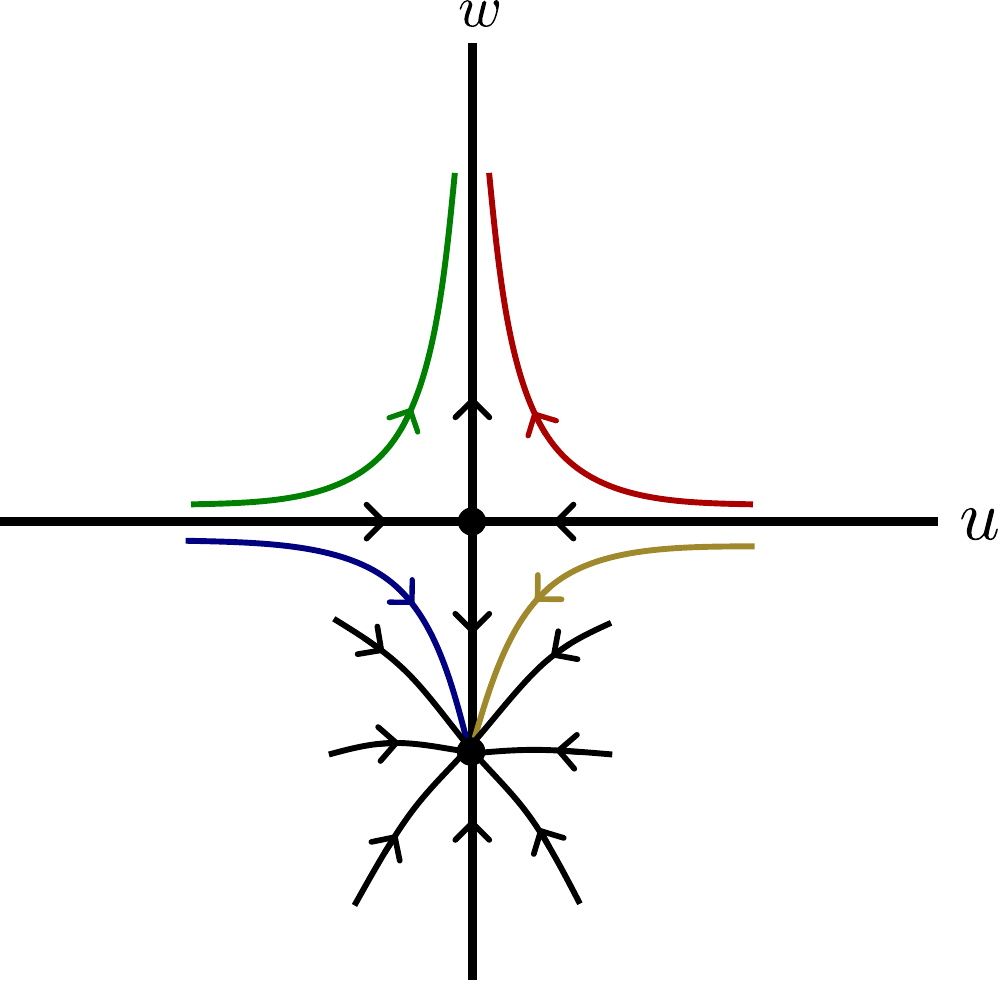}\label{fig1}}
	\hspace{0.8cm}
	\subfigure[]{\includegraphics[scale=0.3]{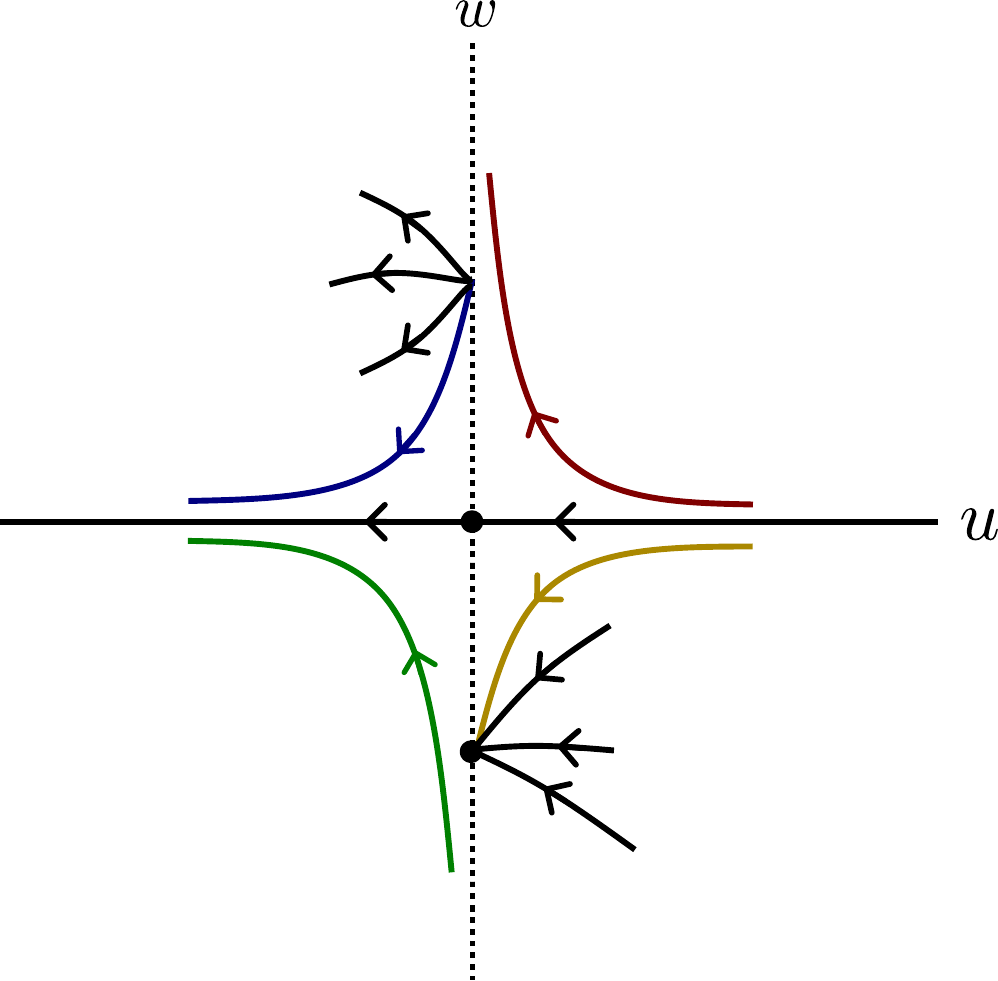}\label{fig2}}	
	\hspace{0.8cm}
\caption{The $u$-directional  blow-up  of the local chart $U_2$ of system  \eqref{blowup-u}.} \label{figure1}
\end{figure}

In order to fully determine the phase portrait of system \eqref{Poincare2},
we do the $v$-directional blow-up. Doing the change of variables   $(u,v)=(zv,v)$ in \eqref{Poincare2}
and eliminating the common factor $v$ through the reparametrization of time $dt/d\tau=v$,
 we obtain that system~\eqref{Poincare2} becomes
\begin{align}\label{blowup-v}
& \dot{z} = - A C z+ C zv - A z^2+ (1-C) z^2v - z^3v,\\
& \dot{v} =Bv^2 - BCv^3- B zv^3,\nonumber
\end{align}
where dot denotes the derivative with respect to $\tau$.
\begin{figure}[h!]
	\centering
	\subfigure[]{\includegraphics[scale=0.3]{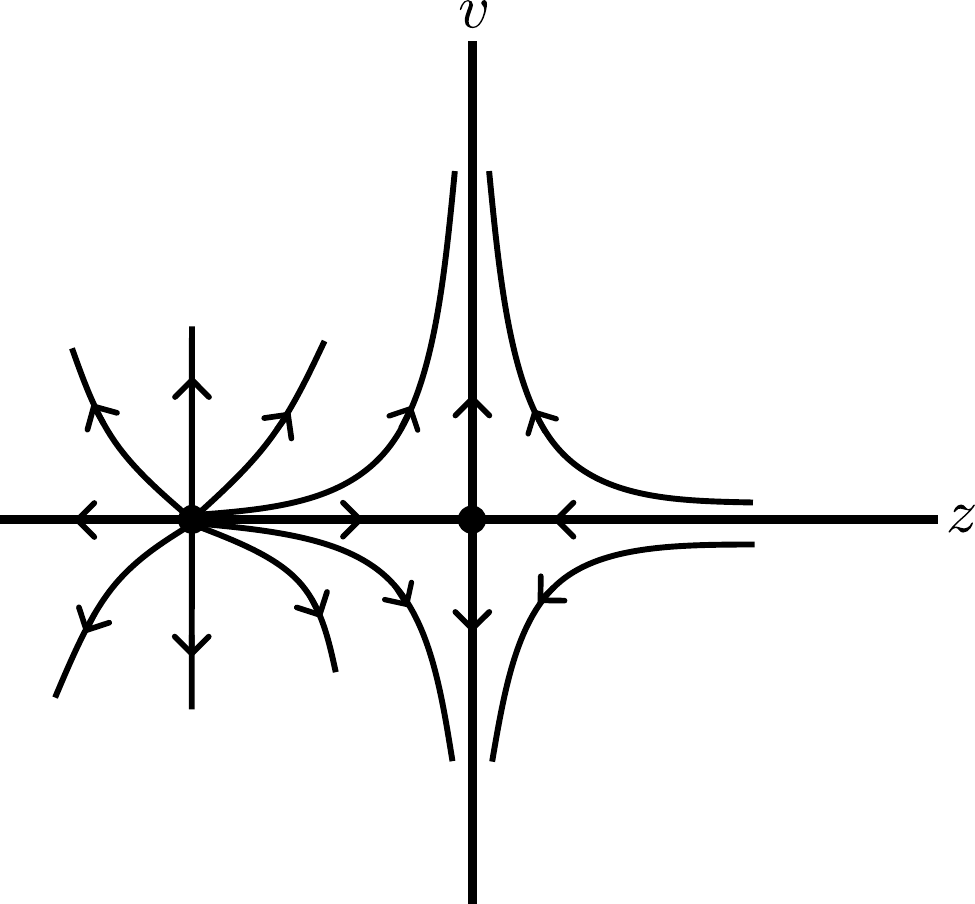}\label{fig21}}
	\hspace{0.8cm}
	\subfigure[]{\includegraphics[scale=0.3]{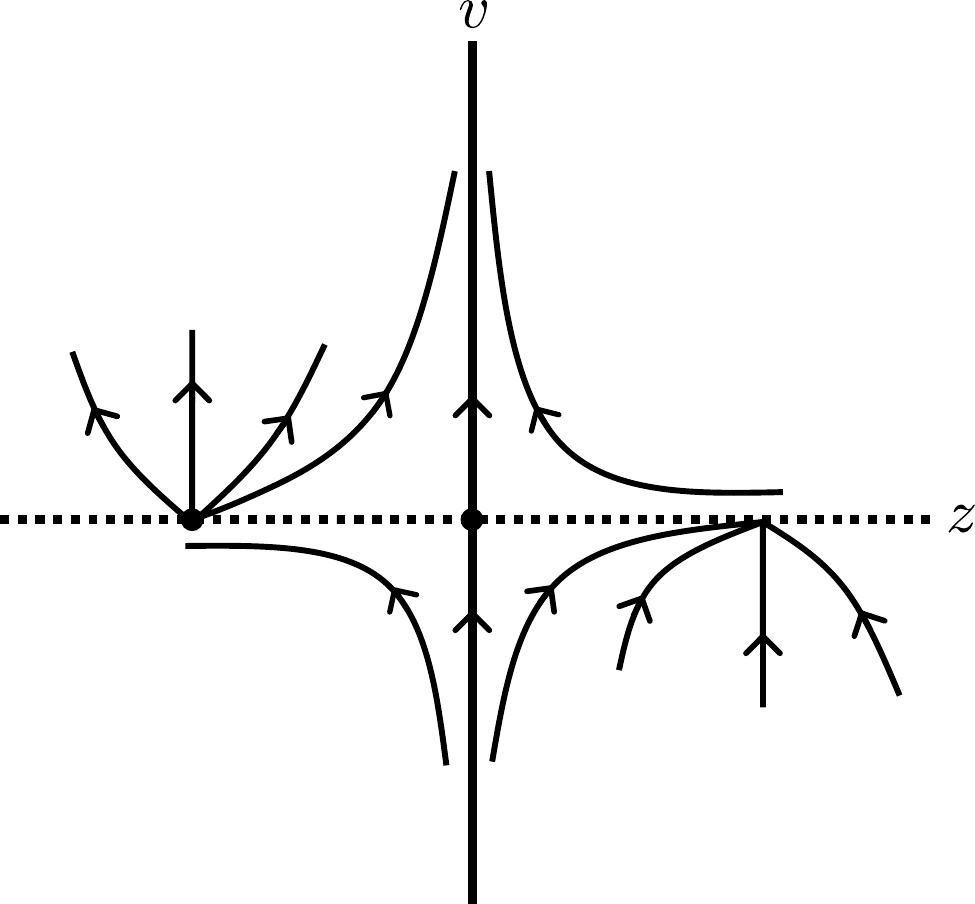}\label{fig22}}	
	\hspace{0.8cm}
\caption{The $v$-directional  blow-up  of the local chart $U_2$ of system  \eqref{blowup-u}.} \label{figure2}
\end{figure}

On  the line $v=0$ system \eqref{blowup-v} has  two equilibrium points: $(0,0)$ and $(-C,0)$.
The linear matrix  of \eqref{blowup-v} at $(0,0)$  has eigenvalues $B>0$ and $-AC<0$ , thus the origin  is a hyperbolic saddle, while 
at  $(-C,0)$ are $B>0$  and  $AC>0$ so it is a hyperbolic unstable node. 
Consequently,  the local phase portrait  of system \eqref{Poincare2} is topologically equivalent to the one given in Figure \ref{figure2}(a).
Back to the plane $(u,v)$ by undoing the change of variable, multiplying by
$v$, the sense of the orbits on the third and fourth quadrants changes and all the points of the $u_2$-axis become singular points.
Also, in the $v$-directional blow-up we must swap the third quadrant and fourth  quadrant, and vice versa.
With these modifications we obtain the phase portrait shown in Figure \ref{figure2}(b). 

Finally, we go back to the differential systems \eqref{Poincare2} in the local chart $U_2$ and the local phase portrait at its
origin is topologically equivalent to the phase portrait of Figure \ref{fig:U2}. Therefore, the origin of the local chart $U_2$
has two parabolic sectors and two hyperbolic sectors. 
\end{proof}

\subsection{Global phase portrait of system \eqref{eq4}}
Now we provide with the topological global phase portrait on the Poincar\'e disc of system \eqref{eq4} using the results obtained in subsections  \ref{pts1} and \ref{pts2}, and the vector field behavior of the system on the coordinate axes.

Combining all these previous information one obtains  the two possible phase portraits of system \eqref{eq4} in the Poincar\'e disc 
corresponding  to quadrant of $\mathbb{R}^2_+$.
We observe that there are two  canonical regions, and the  phase portraits for $1-AC\leq 0$ and $1-AC>0$ are 
topologically equivalent to the ones shown in Figure \ref{fig:global1}, as
is established in  Theorem \ref{Theo2}.

\section{Concluding Remarks}
In this paper, we theoretically derive the global phase portrait of a Leslie-Gower predator-prey model with a linear functional response and a generalist predator. It is well-known that the classical Leslie-Gower model has a unique globally asymptotically stable 
coexistence equilibrium point. Under these modifications, we demonstrate that this global attractor persists only when $r > qc$. Biologically, this means that the intrinsic prey growth rate ($r$) must exceed the product of the maximum per capita consumption rate of the predator ($q$) and the availability of alternative food sources ($c$). Otherwise, when $r \leq  qc$  the coexistence equilibrium point undergoes to a global attractor border equilibrium point,
leading to the extinction of the prey population.

\appendix
\section{Invariant algebraic curves and Liouvillian integrability} \label{ap_inregra}
We consider the following polynomial differential system in $\mathbb{R}^2$, 
\begin{equation} \label{sys_pqq}
\dot{x} = P(x,y), \qquad \dot{y}= Q(x,y),
\end{equation}
where the dot  indicates  differentiation with respect to the time variable $t$. The degree of the polynomial differential system \eqref{sys_pqq} is the maximum value $d$ of the degrees of the polynomials $P$ and $Q$. 

A function $H(x,y)$ is said to be a {\em first integral} of \eqref{sys_pqq} if it is continuous and defined on a full Lebesgue measure subset 
$\Omega\subset {\mathbb R}^2$,  not locally constant on any positive Lebesgue measure subset of $\Omega$ and constant along each orbit in 
$\Omega$ of that system.

The vector field related  to system \eqref{sys_pqq} is defined as
$$
X = P\frac{\partial}{\partial x} + Q\frac{\partial}{\partial y}.
$$
Let $\text{div} (X)= \partial P/\partial x + \partial Q/\partial y$ be the divergence of $X$.

We say that    $f = f (x, y)  = 0$  in $\mathbb{R}^2$   is an {\em  invariant algebraic curve} for the polynomial
differential system \eqref{sys_pqq} if
\begin{equation}\label{eq2_apint}
 P\frac{\partial f}{\partial x} + Q\frac{\partial f}{\partial y} = Kf,
\end{equation}
for some polynomial $K = K(x, y)$, called the {\em cofactor} of the algebraic curve $f=0$. 
From \eqref{eq2_apint} we see  that the degree of the cofactor is at most $d-1$,  and also that the algebraic curve $f = 0$ is formed by trajectories of the polynomial differential system \eqref{sys_pqq}. Therefore, this makes it {\em invariant} with respect to the flow of this system.
 This means that if a trajectory starts on the curve it does not leave it.
 
 Let ${\mathcal E}_m(X)$ be the $m$-th {\em extactic curve} on X, defined by the polynomial equation
 $$
 {\mathcal E}_m(X) = \det \left( 
 \begin{array}{cccc}
 v_1 & v_2 & \cdots & v_l \\
 X( v_1) & X(v_2) & \cdots & X(v_l) \\
 \vdots & \vdots & \cdots & \vdots \\
X^{l -1}( v_1) & X^{l -1}(v_2) & \cdots & X^{l -1}(v_l) 
 \end{array}
 \right) =0,$$
where $v_1,\dots, v_l$ is a basis of ${\mathbb C}_m[x,y]$  (that is, the ${\mathbb C}$ vector space of polynomials in ${\mathbb C}[x,y]$ of degree at most $m$), and then $l =(m+1)(m+2)/2$.

\medskip
\begin{prop}
Every algebraic curve of degree $m$ invariant by the vector field $X$ is a factor of ${\mathcal E}_m(X)$. 
\end{prop}

An invariant algebraic curve $f$ of degree $m$ for the vector field $X$  has  {\em algebraic multiplicity} $k$ when $k$ is the greatest positive integer such that the $k$-th power of $f$ divides ${\mathcal E}_m(X)$.
For  more details about the multiplicity of an invariant curve and other properties of the extactic curve, the reader is referred to \cite{chris}.

The algebraic multiplicity of a curve is related to the so-called {\em exponential factor}. 
Let $f,g\in {\mathbb C}[x,y]$ be two coprime polynomials. It is said that $F=\exp (g/f)$ is an  {\em exponential factor} of the system \eqref{sys_pqq} of degre $d$, if 
$$
P\frac{\partial F}{\partial x} + Q \frac{\partial F}{\partial y} = LF,
$$
for some polynomial $L$ of degree at most $d-1$. This polynomial is  named the {\em cofactor} of the exponential factor $F$.
The quotient $g/f$ is called the {\em exponential coefficient} of $X$.
We note that the exponential factors also provide cofactors and appear when invariant algebraic curves collapse, that is, when they have a multiplicity greater than one, see~\cite{chris}.

Next we present two known propositions, whose proof  and  their geometrical meaning are given in \cite{chris}. 
The first one   characterizes the algebraic multiplicity of an invariant algebraic curve in terms of the number of exponential factors of the system associated with the invariant algebraic curve.  
The second identifies which types of invariant algebraic curves give rise to the exponential factor.

\medskip
\begin{prop} \label{prop_integra}
Let $f=0$ be an irreducible invariant algebraic curve of degree $m$ of the polynomial vector field $X$ with cofactor $K$. Then the algebraic multiplicity of the curve $f=0$ is $k$ if and only if $X$ has $k-1$ exponential factors of the form $\exp(g_j/f_j)$ for $j=1,\dots , k-1$ and the degree of $g_j$ is at most $j m$.
\end{prop}

\medskip
\begin{prop} \label{prop_exp1}
The following statements hold.
\begin{enumerate}
\item[(a)] If $F = \exp(g/f)$ is an exponential factor for the polynomial system \eqref{sys_pqq} and $f$ is not a
constant polynomial, then $f = 0$ is an invariant algebraic curve.
\item[(b)] Eventually, $F=\exp(g)$  can be exponential factors coming from the multiplicity of the infinite invariant straight line.
\end{enumerate}
\end{prop}

Next,  we define a {\em Darboux function} of a vector field $X$  as  a function of the form
\begin{equation} \label{sys_pq2}
D =\prod f_i^{\lambda_i} F_j^{\mu_j},
\end{equation}
where the $f_i=0$ are invariant algebraic curves and $F_j$ are exponential factors of $X$.

The following results come from Darboux, but the present versions  are
proved in~\cite{Dumortier}, they explain how to find Darboux and Liouville first integrals.

\medskip
\begin{thm} \label{teo_app} Suppose that the polynomial system \eqref{sys_pqq}  admits $p$ irreducible invariant algebraic curves 
$f_i=0$ with cofactors $K_i$ and $q$ exponencial factors $F_j$ with cofactors $L_j$. Then, the following statements hod.
\begin{enumerate}
\item[(a)] There exist $\lambda_i$'s and $\mu_i$'s in ${\mathbb C}$, not all zero, such that 
$$\displaystyle{\sum_{i=1}^p \lambda_i K_i + \sum_{j=1}^q \mu_j K_j =0 },$$
if and only if the Darboux function \eqref{sys_pq2} is a first integral of system \eqref{sys_pqq}.
\item[(b)] There exist $\lambda_i$'s and $\mu_i$'s in ${\mathbb C}$, not all zero, such that 
$$\displaystyle{\sum_{i=1}^p \lambda_i K_i + \sum_{j=1}^q \mu_j K_j = -\text{div} (X) },$$
if and only if the Darboux function \eqref{sys_pq2} is an integrating factor of system \eqref{sys_pqq}.
\end{enumerate}
\end{thm}

In order to  prove the result related to the Liouville first integrals we
make use of the following result, proved in  \cite{singer}.

\medskip
\begin{thm} \label{teo_darint} The polynomial differential system \eqref{sys_pqq} has a Liouville
first integral if and only if it has an integrating factor which is a Darboux function. \label{teo_artes}
\end{thm}
This theorem provides us with  the method of Darboux to find all Liouvillian first
integrals.

\section{Poincar\'e compactification in $\mathbb{R}^2$}\label{ap_compac}
In this section we present a description of the Poincar\'e compactification, this device permit us to obtain a  
 description of the phase portrait in the Poincar\'e disc when dealing with a polynomial vector field in order to find equilibrium points in the finite or infinite part of the plane. 
For more information on this subject the reader is addressed to   Chapter 5 of  \cite{Dumortier} or \cite{Cima1990} and references therein.

Let us consider a polynomial vector field $F(x_1,x_2)=(P(x_1,x_2),Q(x_1,x_2))$ associated to the system of differential equations
\begin{equation*}
\frac{dx_1}{dt}=P(x_1,x_2),\qquad
\frac{dx_2}{dt}=Q(x_1,x_2),
\end{equation*}
where $P(x_1,x_2)$ and $Q(x_1,x_2)$ are polynomials in the plane $(x_1,x_2)$ of degrees $d_1$ and $d_2$, respectively. Let $d=\max\{d_1,d_2\}$ be the maximum of the degrees of both polynomials.
In order to study the global dynamics of system of differential equations associated to the vector field $F$,  this polynomial differential system can be extended by a change of coordinates into an analytic differential system defined on $\mathbb{S}^2.$

Consider the 2-dimensional sphere in $\mathbb{R}^3$, namely, 
$\mathbb{S}^2=\{(y_1,y_2,y_3)\in\mathbb{R}^3\,| \, y_1^2+y_2^2+y_3^2=1 \},$
where we identify 
$T_{(0,0,1)} \mathbb{S}^2$  with $\mathbb{R}^2$ where $T_{(0,0,1)} \mathbb{S}^2$  is the tangent space to $\mathbb{S}^2$ at the north pole $(0,0,1)$ through the equality $(y_1,y_2,y_3)=(x_1,x_2,1)$. Under this identification,   from now on,  we only refer to $\mathbb{R}^2$. The sphere $\mathbb{S}^2$ is the union of the north hemisphere $\mathbb{S}^2_+=\left\{(y_1,y_2,y_3)\in\mathbb{S}^2\,|\,y_3>0\right\}$, the south one $\mathbb{S}^2_-=\left\{(y_1,y_2,y_3)\in\mathbb{S}^2\,|\,y_3<0\right\}$ and the equator $\mathbb{S}^1=\left\{(y_1,y_2,y_3)\in\mathbb{S}^2\,|\,y_3=0\right\}$. 

By using the central projection $\Pi: \mathbb{R}^2\to \mathbb{S}^2$, under which  to each point $x$ in $\mathbb{R}^2$ we associate the two points on $\mathbb{S}^2$ that are on line $\mathcal{L}_x$ that pass through $x$ and the origin of coordinates, these two points are antipodal on the sphere. One point lies on $\mathbb{S}^2_+$ and the other in $\mathbb{S}^2_-$. In doing so, the equator $\mathbb{S}^1$ is identified with the infinity in  $\mathbb{R}^2$. Under the differential $D\Pi$ of the central projection we obtain $\widetilde{F}=D\Pi\circ F$  that provide us with two copies of $F$, one on each hemisphere and none of them is defined on the equator $\mathbb{S}^1,$ but we can obtain a vector field defined over the whole $\mathbb{S}^2$. The new vector field $\mathcal{P}(F)$ defined on $\mathbb{S}^2$ is obtained by multiplying  $D\Pi\circ F$ by $y_3^d$, that is, $\mathcal{P}(F)=y_3^d\widetilde F$. This new vector filed defined on the whole $\mathbb{S}^2$ is known as the Poincar\'e compactification of the vector field $F$ defined on the plane $\mathbb{R}^2$.  

In order to study the behavior of solutions associated to the vector field $F$ near infinity in the plane $\mathbb{R}^2$ it becomes necessary to understand how is the flow  of $\mathcal{P}(F)$     near the equator $\mathcal{S}^1.$ 
To make this study we rely on the fact that $\mathbb{S}^2$ is a manifold and we need  to obtain an expression for $\mathcal{P}(F)$ on $\mathbb{S}^2$ in terms of local coordinate charts.  An atlas that covers it  is given by a set of six local coordinate charts $(U_i,\phi_i),$ $(V_i,\psi_i),$ where $U_i=\left\{(y_1,y_2,y_3)\in\mathbb{S}^2\,|\,y_i>0\right\}$, $V_i=\left\{(y_1,y_2,y_3)\in\mathbb{S}^2\,|\,y_i<0\right\}$, for $i=1,2,3.$ Each one of these maps $\phi_i$ and $\psi_i$ are local maps and  the inverse of the central projections on positive and negative hemispheres in each coordinate $y_i$ directions from the origin to the tangent planes  at the pairs of points $(\pm1,0,0)$, $(0,\pm1,0)$ and $(0,0,\pm1)$, respectively.

Expressions for the  local coordinate charts $\phi_1$ and $\psi_1$ are given, respectively~by 
\begin{equation*}
\phi_1(y_1,y_2,y_3)=-\psi_1(y_1,y_2,y_3)=\left(\frac{y_2}{y_1},\frac{y_3}{y_1}\right)=(u,v).
\end{equation*}
on the coordinate chart on hemispheres $U_1$ and $V_1$.
Let us denote by $(u,v)$ to the image of $(y_1,y_2,y_3)$ under $\phi_1$ or $\psi_1$.  Values $(u,v)$ depend on the considered  local coordinate chart.

In the local coordinate chart $(U_1,\phi_1)$, the following expression are obtained for the differential equations
\begin{equation}\label{PoincCampact1}
\frac{du}{dt}=v^d\left(-uP\left(\frac{1}{v},\frac{u}{v}\right)+Q\left(\frac{1}{v},\frac{u}{v}\right)\right),\quad
\frac{dv}{dt}=-v^{d+1}P\left(\frac{1}{v},\frac{u}{v}\right).
\end{equation}
For $(U_2,\phi_2)$, we get
\begin{equation}\label{PoincCampact2}
\frac{du}{dt}=v^d\left(P\left(\frac{1}{v},\frac{u}{v}\right)-uQ\left(\frac{1}{v},\frac{u}{v}\right)\right),\quad
\frac{dv}{dt}=-v^{d+1}Q\left(\frac{1}{v},\frac{u}{v}\right),
\end{equation}
and, for the local coordinate chart $(U_3,\phi_3)$,
\begin{equation*}
\frac{du}{dt}=P\left(u,v\right),\quad
\frac{dv}{dt}=Q\left(u,v\right),
\end{equation*}
as we see, the expression for the local coordinate chart coincides with the original differential equations. This is consequence of the fact that $U_3=\mathbb{S}^2_+$ and $\phi_3=P_1.$ An important observation is that for the pairs of related coordinate charts $(V_i,\psi_i)$ and $(U_i,\phi_i)$, $i=1,2$, the points $(u,v)$ that lie on $\mathbb{S}_1$ are such that $v=0$, that is, the points on the equator have the form $(u,0)$. It follows from \eqref{PoincCampact1} and  \eqref{PoincCampact2} that, on one side, the equator is invariant under the flow defined by $\mathcal{P}(F)$, and on the other, the equilibrium points at infinity for $F$ are the equilibrium points for $\mathcal{P}(F)$ which lie on $\mathbb{S}_1.$ The Poincar\'e disk $\mathbb{D}$ is obtained by projecting $\mathbb{S}^2_+$ on $y_3=0$ through $\pi(y_1,y_2,y_3)=(y_1,y_2)$. 

If we use the local coordinate charts $(V_i,\phi_i), i=1,2,3$, the differential equations obtained are the given through $U_i,\phi_i$, just multiplied by factor $(-1)^{d+1}$. This last factor plays a key role for studying stability of singularities on the sphere $\mathbb{S}^2.$

Finally, in order to study the qualitative behavior at infinity it is enough to make use of $(U_1,\phi_1)$ and $(U_2,\phi_2).$

\subsection*{Acknowledgment}
The authors would like to thank Jaume Llibre   for his helpful suggestions concerning this work.  
We thank to the  anonymous referees for their valuable
comments and suggestions leading to improvement.
\medskip

\noindent{\bf Funding}  J. D. Garc\'ia-Salda\~na was supported by Direcci\'on de Investigaci\'on of the UCSC through project DIREG 01/2024.
M. \'Alvarez-Ram\'{\i}rez   and M. Medina were  supported by Programa Especial de Apoyo a Proyectos de Docencia e Investigaci\'on  2024, CBI-UAMI.

\noindent{\bf Author Contributions} M. \'Alvarez-Ram\'{\i}rez: Conceptualization, Methodology, Validation, Visualization, Writing
- original draft.  D. Garc\'ia-Salda\~na: Conceptualization, Methodology, Validation, Visualization, Writing
- original draft. M. Medina: Conceptualization, Methodology, Validation, Visualization, Writing- original draft.

\end{document}